\documentclass[a4paper,reqno,oneside]{amsart}
\usepackage{a4wide}
\usepackage{amsmath,amsfonts,amssymb,amsthm}
\usepackage[utf8]{inputenc}
\usepackage{ifthen}
\usepackage[style=american]{csquotes}
\usepackage{color}
\usepackage{graphicx}
\usepackage{hyperref}
\usepackage[shortlabels]{enumitem}
\usepackage[english,algoruled,lined,noresetcount,norelsize]{algorithm2e}
\usepackage{framed}
\usepackage{verbatim}
   \usepackage{graphics}

\RequirePackage{hyperref}

\title[A Dirac-type theorem for Berge cycles in random  hypergraphs]{A Dirac-type theorem for Berge cycles\\ in random  hypergraphs}

\author[D.~Clemens]{Dennis Clemens}
\address{Technische Universität Hamburg-Harburg, Institut f\"ur Mathematik, Am Schwarzenberg-Campus 3, 21073 Hamburg, Germany}
\email{dennis.clemens@tuhh.de}
\author[J.~Ehrenm\"uller]{Julia Ehrenm\"uller}
\address{Ehrenm\"uller GmbH, Keselstra\ss e 16, 87435 Kempten, Germany }
\email{julia@ehrenmueller.ai}
\author[Y.~Person]{Yury Person}
\address{Institut f\"ur Mathematik, Technische Universit\"at Ilmenau, 98684 Ilmenau, Germany}
\email{yury.person@tu-ilmenau.de}
\thanks{YP was supported by DFG grant PE 2299/1-1.\\
   An extended abstract of this paper appeared in the Proceedings of Discrete Mathematics Days
    2016 (Barcelona)~\cite{CEP16}.}

\newtheorem{theorem}{Theorem}
\newtheorem{lemma}[theorem]{Lemma}
\newtheorem{proposition}[theorem]{Proposition}
\newtheorem{corollary}[theorem]{Corollary}
\newtheorem{conjecture}[theorem]{Conjecture}
\newtheorem{remark}[theorem]{Remark}

\theoremstyle{definition}
\newtheorem{definition}[theorem]{Definition}

\newcommand{\oldqed}{}
\def\endofClaim{\hfill\scalebox{.6}{$\Box$}}

\newcommand{\Exp}{{\mathbb{E}}}
\newcommand{\Prob}{{\fam=5 P}}

\newcommand{\eps}{\varepsilon}

\newcommand{\vertices}{V^\ast}
 
\newcommand{\cH}{\mathcal H}

\newcommand{\cS}{\mathcal S}
\newcommand{\cP}{\mathcal P}
\newcommand{\PP}{\mathbb{P}}
\newcommand{\EE}{\mathbb{E}}
\newcommand{\NN}{\mathbb{N}}

\newcommand{\dcup}{\dot\cup}
\newcommand{\polylog}{\mathrm{polylog}}

\newcommand{\Hrnp}{H^{(r)}(n,p)}

\DeclareMathOperator{\Bin}{Bin}

\newcommand{\kk}{17r} 

\newcommand{\tr}{\textcolor{red}}
\newcommand{\tb}{\textcolor{blue}}

\date{\today}
\begin{document}
\begin{abstract}
A Hamilton Berge cycle of a hypergraph on $n$ vertices is an alternating sequence $(v_1, e_1, v_2, \ldots, v_n, e_n)$ of distinct vertices $v_1, \ldots, v_n$ and distinct hyperedges $e_1, \ldots, e_n$ such that $\{v_1,v_n\}\subseteq e_n$ and $\{v_i, v_{i+1}\} \subseteq e_i$ for every $i\in [n-1]$. We prove the following Dirac-type theorem about Berge cycles in the binomial random $r$-uniform hypergraph $H^{(r)}(n,p)$: for every integer $r \geq 3$, every real $\gamma>0$ and 
$p \geq \frac{\ln^{\kk} n}{n^{r-1}}$ asymptotically almost surely,  every spanning subgraph $H \subseteq H^{(r)}(n,p)$ with  minimum vertex degree $\delta_1(H) \geq \left(\frac{1}{2^{r-1}} + \gamma\right) p \binom{n}{r-1}$ contains a Hamilton Berge cycle. The minimum degree condition is asymptotically tight and the bound on $p$ is optimal up to some polylogarithmic factor.  
\end{abstract}
\maketitle

\section{Introduction}

Many classical theorems of extremal graph theory give sufficient optimal minimum degree conditions 
for graphs to contain copies of large or even spanning structures. 
Lately it became popular to phrase such extremal results in terms of local resilience, where the \emph{local resilience} of a graph $G$ with respect to a given monotone increasing graph property $\cP$ is defined as the minimum number $\rho\in \mathbb R$ such that one can obtain a graph without property $\cP$ by deleting at most $\rho \cdot \deg(v)$ edges from every vertex $v\in V(G)$. For instance, using this terminology, Dirac's theorem~\cite{dirac1952} says that the local resilience of the  complete graph $K_n$ with respect to Hamiltonicity is $1/2-o(1)$. 

In recent years, an active and fruitful research direction in extremal and probabilistic combinatorics has become the study of resilience of random and pseudorandom structures. The systematic study of those with respect to various graph properties was initiated by Sudakov and Vu in~\cite{SudVu}, who in particular proved that  $G(n,p)$ (i.e.~the Erd\H{o}s-Renyi random graph model that is defined on the vertex set $[n]$ with each pair of vertices forming an edge randomly and independently with probability $p$)  has resilience at least $1/2-o(1)$ with respect to Hamiltonicity a.a.s.~for  $p> \ln^4n/n$.
 This result was improved by  Lee and Sudakov~\cite{lee2012dirac} to $p\gg \ln n/n$, which is essentially best possible with respect to both the constant $1/2$ and the edge probability, since one can find a.a.s.~disconnected spanning subgraphs of $G(n,p)$ with degree at most $(1/2-o(1))pn$ and since $G(n,p)$ itself is a.a.s.~disconnected for  $p\le (1-o(1))\ln n/n$.

A lot of resilience results are known for random graphs. For instance, the containment of triangle factors~\cite{BLS12}, almost spanning trees of bounded degree~\cite{BalCsaSam11}, pancyclic graphs~\cite{KLS10}, almost spanning and spanning bounded degree graphs with sublinear bandwidth~\cite{allenBT,BKT13,HLS12}, directed Hamilton cycles~\cite{ferber2014robust,HStSu16}, perfect matchings and Hamilton cycles in random graph processes~\cite{NST18}, almost spanning  powers of  cycles~\cite{SST18}  were studied. 

 An \emph{$r$-uniform hypergraph} is a tuple $(V,E)$ with $E\subseteq \binom{V}{r}$ and thus the generalisation of a graph: the elements of $V$ are called \emph{vertices} and the elements of $E$ \emph{hyperedges} (or \emph{edges} for short).  It is therefore natural to ask for degree conditions that force a subhypergraph (or \emph{subgraph} for short) of the complete hypergraph to contain a copy of some given large structure. Such problems  have been studied extensively in the last years, especially for different kinds of Hamilton cycles. Furthermore, (bounds on) the threshold for the existence of a Hamilton cycle in the random $r$-uniform hypergraph model $\Hrnp$ (every possible edge appears with probability $p$ independently of the others) have been determined for various notions of cycles.  We refer to~\cite{KO14} for an excellent survey by K\"uhn and Osthus of such problems.

The purpose of this work is to provide a Dirac-type result in the random hypergraph $\Hrnp$. This result was announced and its proof sketched in~\cite{CEP16} and is,  
to the best of our knowledge, the first local resilience result in random hypergraphs at an almost optimal edge probability. The only other resilience result in random hypergraphs is a recent work of Ferber and Hirschfeld~\cite{FH18} on the resilience of perfect matchings with respect to the codegree condition in random hypergraphs  at the asymptotically optimal probability.

Given an $r$-uniform hypergraph $H=(V,E)$. We use $\deg_H(v)$ to denote the vertex degree of a vertex $v$ in $H$, i.e.\ the number of hyperedges of $H$ that contain $v$. The minimum vertex degree of a hypergraph $H$ is then $\delta_1(H):=\min_{v} \deg_H(v)$. We will also consider other degree notions such as $\deg_H(T)$ defined as $|\{e\colon e\supseteq T, e\in E(H)\}|$, i.e.\ the number of edges that contain a given tuple $T$. Generally, we define 
the maximum ($\Delta_\ell(H)$) $\ell$-collective degree as follows: 
 \[
  \Delta_\ell(H):=\max_{T\in \tbinom{V}{\ell}}\deg_H(T).
 \]
   The notion of resilience in graphs extends verbatim to  the setting of hypergraphs.

We will be interested in resilience results of  random $r$-uniform hypergraphs with respect to weak and Berge Hamiltonicity. Weak and Berge cycles are the earliest notion of cycles in hypergraphs and are defined as follows.

\begin{definition}
A  \emph{weak cycle} is an alternating sequence $(v_1, e_1, v_2, \ldots, v_k, e_k)$ of distinct vertices $v_1, \ldots, v_k$ and hyperedges $e_1, \ldots, e_k$ such that $\{v_1,v_k\}\subseteq e_k$ and $\{v_i, v_{i+1}\} \subseteq e_i$ for every $i\in [k-1]$. A weak cycle is called \emph{Berge cycle} if all its hyperedges are distinct.
\end{definition}

If $P= (v_1, e_1, v_2, \ldots, v_n, e_n)$ is a weak cycle or a Berge cycle in a hypergraph $H$ on $n$ vertices, then $P$ is called \emph{weak Hamilton cycle} or \emph{Hamilton Berge cycle} of $H$, respectively. Other common notion of cycles are $\ell$-cycles. For an integer $1 \leq\ell\leq r$, an $r$-uniform hypergraph $C$ is an $\ell$-cycle if there exists a cyclic ordering of the vertices of $C$ such that every hyperedge of $C$ consists of $r$ consecutive vertices and such that every pair
of consecutive hyperedges intersects in precisely $\ell$ vertices. 

If $\ell =1$, then $C$ is called a \emph{loose cycle} and if $\ell = r-1$, then $C$ is called a \emph{tight cycle}. 
Observe that every $\ell$-cycle is a Berge cycle. Furthermore, every tight Hamilton cycle (i.e.~a spanning tight cycle) is a Hamilton Berge cycle. This is however not true for $\ell$-cycles if $\ell < r-1$ since a Hamilton $\ell$-cycle in a hypergraph on $n$ vertices has $n/(r-\ell)$ hyperedges, whereas a Hamilton Berge cycle has $n$ hyperedges. But since hyperedges may be repeated in weak cycles, a Hamilton $\ell$-cycle is a weak Hamilton cycle. There is also an extensive literature on Dirac-type results  that are stated in terms of the minimum degree for dense graphs and hypergraphs, we refer to the illuminating surveys~\cite{KO14,rodl2010dirac,zhao2015recent}. 

%

Surprisingly, until recently,  the only result on the minimum vertex  degree which implies the existence of a weak or a Berge Hamilton cycle was the one due to Bermond, Germa, Heydemann, and Sotteau~\cite{bermond1976}. They proved that for every integer $r\geq 3$ and $k \geq r+1$ any $r$-uniform hypergraph $H$ with minimum vertex degree $\delta_1(H) \geq \binom{k-2}{r-1}+ r -1$ contains a Berge cycle on at least $k$ vertices. 
If we ask for a  Hamilton Berge cycle in an $r$-uniform hypergraph on $n$ vertices, where $r$ is fixed and $n$ is large, then the bound $\binom{n-2}{r-1}+r-1$ is very weak since it differs from the maximum possible degree by $\binom{n-2}{r-2}-r+1$. 
Certainly, the two propositions below are folklore and should be known.
\begin{proposition}\label{thm:Diracweak}
Let $r\geq 3$ and $n \geq r$ and let $H$ be an $r$-uniform hypergraph on $n$ vertices. If 
$\delta_1(H) > \binom{\lceil n/2\rceil-1}{r-1},$ then $H$ contains a weak Hamilton cycle. 
\end{proposition}
The proof is a one-line argument by replacing every edge of $H$ with a clique on $r$ vertices and applying the original theorem by Dirac. The bound on the minimum vertex degree is sharp. Indeed, for even $n$, the disjoint union of two copies of the complete $r$-uniform hypergraph $K_{\frac{n}{2}}^{(r)}$ on $\frac{n}{2}$ vertices has minimum vertex degree $\binom{n/2-1}{r-1}$ but is disconnected. For odd $n$, the hypergraph $H$ on $n$ vertices that is the composition of two copies of $K_{\lceil \frac{n}{2} \rceil}^{(r)}$ that share one vertex satisfies $\delta_1(H) =  \binom{\lceil n/2\rceil-1}{r-1}$ but does not contain a weak Hamilton cycle.

The following result can be obtained along the lines of the proof of Dirac's theorem for graphs.
\begin{proposition}\label{thm:DiracBerge}
Let $r \geq 3$ and let $H$ be an $r$-uniform hypergraph on $n>2r-2$ vertices. If 
$\delta_1(H) \ge  \binom{\lceil n/2 \rceil -1}{r-1}+ n-1$ then $H$ contains a Hamilton Berge cycle.  
\end{proposition}
In any case, it follows from Propositions~\ref{thm:Diracweak} and~\ref{thm:DiracBerge} that the resilience of the complete hypergraph $K^{({r})}_n$ is $1-2^{1-r}-o(1)$ with respect to both weak and Berge Hamiltonicity. Recently, Coulson and Perarnau~\cite{CP18} improved the lower bound from Proposition~\ref{thm:DiracBerge} for sufficiently large $n$ to the optimal 
$\delta_1(H) \ge  \binom{\lceil n/2 \rceil -1}{r-1}$.

%

Like in the setting of graphs, a natural question is which sparse random hypergraphs contain a weak Hamilton cycle or even a Hamilton Berge cycle and how robust these hypergraphs are with respect to these properties? Recall that by $\Hrnp$ we denote the random $r$-uniform hypergraph model on the vertex set $[n]$, where each set of $r$ vertices forms an edge randomly and independently with probability $p=p(n)$. Poole~\cite{poole2014weak} determined the threshold for the existence of a weak Hamilton cycle in $\Hrnp$.

\begin{theorem}[Theorem~1.1 in~\cite{poole2014weak}]\label{thm:poole}
Let $r\geq 3$. Then
\[\Prob\big[H^{(r)}(n,p) \text{ is weak Hamiltonian}\big] \to 
\begin{cases}
0 & \text{if } p \leq (r-1)!\frac{\ln n - \omega(1)}{n^{r-1}} \\
e^{-e^{-c}} & \text{if } p= (r-1)!\frac{\ln n + c_n}{n^{r-1}} \\
1 & \text{if } p \geq (r-1)!\frac{\ln n + \omega(1)}{n^{r-1}} \ ,
\end{cases}\]
where 
 $c_n$ is any function tending to $c\in \mathbb R$  and 
 and $\omega(1)$ is an arbitrary function that tends to infinity. 
\end{theorem}

Since every Hamilton Berge cycle is in particular a weak Hamilton cycle, Theorem~\ref{thm:poole} yields that $H^{(r)}(n,p)$ a.a.s.~does not contain a Hamilton Berge cycle if $p \leq (r-1)!\frac{\ln n - \omega(1)}{n^{r-1}}$. Suprisingly, before our extended abstract~\cite{CEP16} an upper bound on the threshold of $H^{(r)}(n,p)$ being Berge Hamiltonian hasn't been studied directly, but recently, Bal and Devlin~\cite{BD18} determined the threshold for Hamilton Berge cycles up to the constant factor.


\subsection*{Main result}
In this paper we prove the following Dirac-type result for the existence of Hamilton Berge cycles in random hypergraphs.

\begin{theorem}\label{berge:thm:main}
For every integer $r\geq 3$ and every real $\gamma >0$ the following holds asymptotically almost surely~for $\mathcal H = H^{(r)}(n,p)$ if $p \geq \frac{\log^{\kk} n}{n^{r-1}}$. Let $H \subseteq \mathcal H$ be a spanning subgraph with $\delta_1(H) \geq \left(\frac{1}{2^{r-1}} + \gamma\right) p \binom{n}{r-1}$. Then $H$ contains a Hamilton Berge cycle. 
\end{theorem}

It is worth noting that it is not possible to merely reduce the problem of finding Berge cycles to a problem of finding a Hamilton cycle in the random graph to achieve the same resilience $1-2^{1-r}+o(1)$ in the random hypergraph $\Hrnp$. One possibility of such a `reduction' would be to declare an edge $uv\in \binom{[n]}{2}$ to lie in $G$ if $\{u,v\}$ lies in some hyperedge of $\Hrnp$. Then such $G$ behaves much like the random graph $G(n,p')$, where $1-p'=(1-p)^{\binom{n-2}{r-2}}$, i.e.\ $p'=\Omega(\polylog(n)/n)$. However, such a reduction leads to a resilience which is far  from being the asymptotically optimal one, as asserted by our Theorem~\ref{berge:thm:main}.

The minimum degree condition is asymptotically tight, meaning that we cannot replace it with 
$\delta_1(H) \geq \left(\frac{1}{2^{r-1}} - \gamma\right) p \binom{n}{r-1}$. Indeed, given
$\mathcal H = H^{(r)}(n,p)$ together with a partition $V(\mathcal H)=V_1\cup V_2$ with $|V_1-V_2|\leq 1$,
chosen uniformly at random among all such partitions, it happens a.a.s.\ that, for $i=1,2$,
the degree of every $v\in V_i$ into $V_i$ is at least 
$\left(\frac{1}{2^{r-1}} - \gamma\right) p \binom{n}{r-1}$. Fixing a partition with that property,
we find the hypergraph $H:=\mathcal H [V_1]\cup \mathcal H [V_2]$ which does not contain
a Berge Hamilton cycle and which satisfies
$\delta_1(H) \geq \left(\frac{1}{2^{r-1}} - \gamma\right) p \binom{n}{r-1}$.

Furthermore, the bound on $p$ is optimal up to possibly this polylogarithmic factor, and  it provides an alternative proof of the result in~\cite{poole2014weak} with only slightly weaker edge probability. 



The proof  is based on the absorbing method developed by R{\"o}dl, Ruci{\'n}ski, and Szemer{\'e}di~\cite{RRSz06}. Of particular importance are the ideas from the proof of a Dirac-type result for random directed graphs due to Ferber, Nenadov, Noever, Peter and {\v{S}}koric~\cite{ferber2014robust}, which allow us to apply this method in such a very sparse scenario. 

\subsection*{Organization of the paper}
The rest of the paper is organised as follows. In Section~\ref{sec:probability} we state concentration inequalities that we are going to use. In Section~\ref{sec:further_tools} we introduce some technical definitions, the notion of $(\eps,p)$-pseudorandom hypergraphs and prove some basic properties around these notions. In Section~\ref{sec:absorbers} we introduce the central concept of absorbers and in Section~\ref{sec:connection} we prove a connection lemma which will allow us to connect many pairs of vertices by disjoint paths. We provide in Section~\ref{sec:main_theorem} first an outline of the proof of Theorem~\ref{berge:thm:main} and then we prove a technical theorem about Hamilton Berge cycles in pseudorandom hypergraphs which implies Theorem~\ref{berge:thm:main}. We present the proof of Proposition~\ref{thm:DiracBerge}  in Section~\ref{sec:Dirac}. 

\section{Probabilistic tools: concentration inequalities}\label{sec:probability}
In our proofs we use the following standard bounds on deviations of random variables. Their proofs can be found in e.g.~\cite{janson2011random}. 
%
%
%
The following two bounds belong to Chernoff's inequality, which collects different exponentially decreasing bounds on the tails of 
a binomial distribution. 

\begin{theorem}[Chernoff's inequality I]\label{thm:chernoff}
For every random variable $X \sim \Bin(n,p)$ and every $\eps \leq 3/2$ we have 
\[\Prob\big[|X-\Exp[X]| > \eps \Exp[X]\big] < 2 \exp\left(-\frac{\eps^2 \Exp[X]}{3}\right).\]
\end{theorem}

The second Chernoff's inequality that we need provides only a bound on the upper tail of the binomial distribution. 

\begin{theorem}[Chernoff's inequality II]\label{chernoff1}
For every random variable $X\sim \Bin(n,p)$ and every $t\geq 0$ we have
$$\Prob\big[X\geq  \Exp[X]+t\big]\leq \exp\left(-\frac{t^2}{2(\Exp[X] + t/3)}\right). $$
\end{theorem}

Finally we consider binomial random subsets. For $\Gamma=[n]$ let $\Gamma_{p_1, \ldots, p_n}$ be defined by including for every $i\in [n]$ the $i$-th element of $\Gamma$ with probability $p_i$ independently of all other elements of $\Gamma$. For each set $\cS \subseteq 2^{\Gamma}$ of subsets of $\Gamma$ and each set $A\in \cS$, we let $X_A$ denote the indicator variable for the event $A \subseteq \Gamma_{p_1, \ldots, p_n}$. Janson's inequality gives an exponentially small bound on the lower tail of the distribution of sums of such indicator variables. 

\begin{theorem}[Janson's inequality]\label{thm:janson}
Let $\Gamma$ be a finite set and let $\cS \subseteq 2^{\Gamma}$ be a set of subsets of $\Gamma$. If $X=\sum_{A\in {\cS}} X_A$, where $X_A$ is an indicator variable, and $0\leq t\leq \Exp[X]$,
then
\[
\Prob\big[X\leq \Exp[X]-t\big]\leq \exp\left(-\frac{t^2}{2\overline{\Delta}}\right),
\]
where 
\[
\overline{\Delta}=\Exp[X] + \sum\limits_{A\in{\cS}\atop} \sum\limits_{B\in{\cS}: \atop A\cap B\neq \varnothing, A\neq B}   \Exp[X_AX_B]. 
\]
\end{theorem}

\section{More definitions and auxiliary lemmas}\label{sec:further_tools}
Given an $r$-uniform hypergraph $H$, a subset $U\subseteq V(H)$ and a vertex $u\in V(H)$, we write $\deg_H(u,U)$ to denote the degree of $u$ in $U$, which is the number of edges from $H$
which contain $u$ and which are contained completely in $U\cup\{u\}$. For disjoint $T$ and $U$, we write $e_H(T, \binom{U}{r-1})$  to denote the number of edges $e$ from $H$ with $|e\cap T|=1$ and $|e\cap U|=r-1$. The hypergraph $H$ is sometimes omitted when it is clear from the context.

Our proof will use some central properties of the random hypergraph $\Hrnp$ which we call pseudorandom. So, our main theorem will state that $r$-uniform hypergraphs that are pseudorandom satisfy a Dirac-type theorem about  Hamilton  Berge cycles.

\subsection{Pseudorandom hypergraphs}
The following proposition asserts that for an edge probability 
$p\leq \polylog(n)/n^{r-1}$,
the codegree of a random hypergraph doesn't get too large.
\begin{proposition}\label{prop:codegree}
For every integer $r\ge 3$, every real $c>0$ and $p\le \frac{\ln^c n}{n^{r-1}}$,  with probability at least $1-1/n$, the hypergraph $\Hrnp$ has maximum $2$-collective degree at most $2\ln n$.
\end{proposition}
\begin{proof}
The statement is an immediate consequence of Chernoff's inequality, Theorem~\ref{chernoff1}, and the union bound over $\binom{n}{2}$ possible pairs of vertices. 
\end{proof}

Next we verify that the edges in the random hypergraph $\Hrnp$ are distributed as expected. 

\begin{lemma}\label{lem:upper_disc}
For every integer $r\geq 3$, reals $\eps$, $p>0$ and sufficiently large integer $n$  the following holds  with probability at least $1-\tfrac{2}{n}$ for  any pair of disjoint subsets $T$ and $U$ of $[n]$ with $|U|\le |T|$:
\[
e_{\Hrnp}\left(T,\binom{U}{r-1}\right)\le (1+\eps)p|T|\binom{|U|}{r-1}+|T|\ln^{1+\eps} n.
\]
\end{lemma}
\begin{proof}
Consider the random variable $X=e_{\Hrnp}\left(T,\binom{U}{r-1}\right)$, then its expectation is clearly: $\EE[X]=p|T|\binom{|U|}{r-1}$. 
If $\EE[X]\geq \frac{12}{\eps^2}|T|\ln n$,
then Chernoff's inequality, Theorem~\ref{chernoff1}, with $t=\eps \EE[X]$ gives
\[
\PP\left[X\ge (1+\eps)\EE[X]\right]\le 
\exp\left(-\frac{\eps^2}{3} \EE[X] \right)
\le n^{-4|T|}.
\]
If $\EE[X]< \frac{12}{\eps^2}|T|\ln n$,
then Chernoff's inequality, Theorem~\ref{chernoff1}, with $t=|T| \ln^{1+\eps} n$ gives
\[
\PP\left[X\ge \EE[X]+t\right]\le 
\exp\left(- |T| \ln^{1+\eps} n \right)
\le n^{-4|T|}.
\]
Since there are at most $\binom{n}{|T|}\binom{n}{|U|}\le n^{2|T|}$ pairs $(T,U)$ with fixed sizes satisfying $|T|\ge |U|$ we use union bound to obtain that with probability at least $1-\sum_{i=1}^n n\cdot n^{2i} n^{-4i}\ge 1-\tfrac{2}{n}$ for all pairs of disjoint sets $(T,U)$ with $|T|\ge |U|$ it holds that
\[
e_{\Hrnp}\left(T,\binom{U}{r-1}\right)\leq (1+\eps)p|T|\binom{|U|}{r-1}+|T|\ln^{1+\eps} n.
\]
\end{proof}

\begin{lemma}\label{lem:upper_disc_II}
For every integer $r\geq 3$, reals $\eps,p>0$ with $\eps\in (0,3/2)$ and sufficiently large integer $n$,  the following holds  with probability at least $1-\tfrac{2}{n}$ for  any pair of disjoint subsets $T$ and $U$ with $\eps|U|\le |T| \le |U|\le n/2$ and $|U|\ge m:=\left(\frac{13 (r-1)!\ln n}{\eps^{3}p}\right)^{1/(r-1)}$:
\[
e_{\Hrnp}\left(T,\binom{U}{r-1}\right)\le (1+\eps)p|T|\binom{|U|}{r-1}.
\]
\end{lemma}
\begin{proof}
Consider the random variable $X=e_{\Hrnp}\left(T,\binom{U}{r-1}\right)$, then its expectation is clearly: $\EE[X]=p|T|\binom{|U|}{r-1}$. By Chernoff's inequality, Theorem~\ref{thm:chernoff}, we have for fixed disjoint sets $U$  and $T$ with $|T|\ge \eps |U|\ge \eps m$:
\[
\PP\left[X\ge (1+\eps)\Exp[X]\right]\le 2\exp\left(-\frac{\eps^2 \EE[X]}{3}\right)\le 
2\exp\left(-\eps^{3}p |U|^r/(4(r-1)!)\right)\le n^{-3|U|}.
\]
Since there are at most $\binom{n}{|T|}\binom{n}{|U|}\le \left(en/|U|\right)^{2|U|}$ pairs $(T,U)$ of fixed sizes $|T|$ and $|U|$ we use union bound to obtain that with probability at least $1-\sum_{i=m}^{n/2}\sum_{j=\eps i}^{i} (en/i)^{2i} n^{-3i}\ge 1-\tfrac{2}{n}$ for all pairs of disjoint sets $(T,U)$ it holds that
\[
e_{\Hrnp}\left(T,\binom{U}{r-1}\right)\leq (1+\eps)p|T|\binom{|U|}{r-1}.
\]
\end{proof}

Lemmas~\ref{lem:upper_disc} and~\ref{lem:upper_disc_II} motivate the following definition of an $(\eps,p)$-pseudorandom hypergraph.

\begin{definition}[$(\eps,p)$-pseudorandomness]
Given $p>0$ and $\eps\in(0,3/2)$, a hypergraph $\cH$ on $n$ vertices is $(\eps,p)$-pseudorandom if it satisfies the following properties:
\begin{enumerate}[(i)]
\item \label{psrand:i} for any pair of disjoint subsets $T$ and $U$ of $[n]$ with $|U|\le |T|$ holds:
 \[
 e_{\cH}\left(T,\binom{U}{r-1}\right)\le (1+\eps)p|T|\binom{|U|}{r-1}+|T|\ln^{1+\eps} n~ ,
 \]
\item \label{psrand:ii} for  any pair of disjoint subsets $T$ and $U$
 with $\eps|U|\le |T| \le |U|\le n/2$ and $|U|\ge \left(\frac{13 (r-1)!\ln n}{\eps^{3}p}\right)^{1/(r-1)}$ holds:
\[
e_{\Hrnp}\left(T,\binom{U}{r-1}\right)\le (1+\eps)p|T|\binom{|U|}{r-1}.
\]
\end{enumerate}
\end{definition}

\subsection{A sampling lemma}
We will repeatedly use the fact that in a hypergraph of high minimum degree, a random subset of vertices inherits high minimum degree of every vertex. More precisely, we prove the following.
\begin{lemma}\label{lem:sampling}
Let $r\geq 3$, reals $\gamma>0$, $c'>0$ and $c>2+2c'$ be given. Let $H=(V,E)$  be an $r$-uniform hypergraph on $n$ vertices and  $V'$ be a subset of $V$ with at least $2m$ vertices, where  $m\ge \frac{n}{\ln^c n}$. Furthermore assume that  
$\deg_H(v,V')\ge \gamma  p\binom{|V'|}{r-1}$ for all $v\in V$ and $\Delta_2(H)\le 2\ln n$, where 
$p \geq \frac{\ln^{c r} n}{n^{r-1}}$. Then the following holds for all $n$ sufficiently large.  

There exists a set $U\subset V'$ of size $m$ such that
\begin{enumerate}[(i)]
\item $\deg_H(v,U)\ge (1-\tfrac{1}{\ln^{c'} n})\gamma p\binom{|U|}{r-1}$ for all $v\in V$, and
\item $\deg_H(v,V'\setminus U)\ge (1-\tfrac{1}{\ln^{c'} n})\gamma p\binom{|V'|-|U|}{r-1}$ for all $v\in V$.
\end{enumerate}
\end{lemma}
\begin{proof}
We choose a set $U$ randomly, by including every vertex $u$ from $V'$ into $U$ with probability $q=\tfrac{m}{n'}$ independently, where we set $n':=|V'|$. Thus, we associate with every vertex $u$ a Bernoulli variable $t_u$ with parameter $q$. 

For a given vertex $v\in V$ let $X_v$ be the random variable for $\deg_H(v,U)$. Clearly, we can write 
$X_v$ as the following sum of indicator random variables:
 \[
 X_v=\sum_{A\in \cS} X_{A}, 
 \]
 where $X_{A}=\prod_{u\in A}t_u$ and $\cS=\{e\setminus\{v\}\colon e\in E(H),  v\in e \}$.

We are going to apply Theorem~\ref{thm:janson}  and for this we put the following estimates: 
$\EE[X_v]=q^{r-1}\deg_H(v)\ge q^{r-1}\gamma  p\binom{n'}{r-1}\ge \gamma\ln^{c} n/(2(r-1)!)$, 
and 
\begin{align*}
\overline{\Delta}
& =\Exp[X_v] + \sum\limits_{A\in{\cS}\atop} \sum\limits_{B\in{\cS}: \atop A\cap B\neq \varnothing, A\neq B}   \Exp[X_AX_B] \\
& \le \EE[X_v]+\deg_H(v)(r-1)\Delta_2(H)q^{r}\le \EE[X_v](1+2 rq\ln n).
\end{align*}

Then Theorem~\ref{thm:janson} yields
\[
\PP\left[X_v\le \EE[X_v]-\tfrac{1}{\ln^{c'}n}\EE[X_v]\right]
	\le \exp\left(-\tfrac{\EE[X_v]^2}{2(\ln^{2c'}n)\EE[X_v](1+2 rq\ln n)}\right)
	\leq \exp\left(-\tfrac{\EE[X_v]}{4r \ln^{1+2c'} n}\right)
	<n^{-3}.
\]

A similar argument applies also to $V'\setminus U$. Taking union bound, with probability, say, at least $1-\tfrac{1}{n}$ the properties (i) and (ii) are satisfied by all vertices $v\in V$.

On the other hand, $\PP[|U|=m]=\PP[|V'\setminus U|=n'-m]=\binom{n'}{m}q^m(1-q)^{n'-m}=\frac{(1+o(1))}{\sqrt{2\pi q(n'-m)}}\ge (1+o(1))\sqrt{\frac{2}{\pi n'}}$. Therefore, with positive probability (at least $(1+o(1))\sqrt{\frac{2}{\pi n'}}-\tfrac{1}{n}$) there exists a desired set $U$.
\end{proof}

\subsection{Matchings}
Our building blocks for Hamilton cycles will consist of collections of edges between pairs of equal-sized sets, which we will refer to as $(U_1,U_2)$-matchings. Moreover, these edges will intersect both sets $U_1$ and $U_2$ in a clearly specified way. The following two definitions make these ideas precise.

\begin{definition}[An $(i,j)$-edge for $(U_1,U_2)$]
Given an $r$-uniform hypergraph $H=(V,E)$, two disjoint subsets $U_1$, $U_2\subseteq V$ and an edge $e\in E$. We call $e$ an $(i,j)$-edge for $(U_1,U_2)$, if $|e\cap U_1|=i$ and $|e\cap U_2|=j$ hold.
\end{definition}

\begin{definition}[$(U_1,U_2)$-matching in $H$]
Given an $r$-uniform hypergraph $H=(V,E)$ and two disjoint subsets $U_1$, $U_2\subseteq V$ with $|U_2|\geq |U_1|=m$. We call the set $M=\{e_1,\ldots,e_m\}$  a $(U_1,U_2)$-matching in $H$ if there exists a matching $M'=\{a_ib_i\colon a_i\in U_1, b_i\in U_2, i\in[m]\}$ in the complete bipartite graph $K_{U_1,U_2}$ with classes $U_1$ and $U_2$ such that $a_i$, $b_i\in e_i$ and $e_i$ is a $(1,r-1)$-edge or a $(r-1,1)$-edge for $(U_1,U_2)$ for every $i\in [m]$. 

We call the vertices $a_i$ and $b_i$ the \emph{endpoints of the matching edge $e_i$}.
\end{definition}

The next lemma asserts that between two disjoint subsets $U_1$, $U_2$ of vertices of high `minimum degree' and of size $m=n/\polylog(n)$ there must always be a $(U_1,U_2)$-matching $M$ which intersects only $U_1$ and  $U_2$ in the `pattern' $(1,r-1)$ or $(r-1,1)$.  Its proof is an application of Hall's matching criterion under the exploitation of  the properties of $(\eps,p)$-pseudorandom hypergraphs.

\begin{lemma}\label{lem:matching}
For every integer $r\geq 3$, every real $\gamma \in(0,1)$ and $c>1$ there exists an $\eps>0$ such that  the following holds for any $(\eps,p)$-pseudorandom $r$-uniform hypergraph $\cH$ on $n$ vertices with  $n$ sufficiently large and $p \geq \frac{\ln^{c r} n}{n^{r-1}}$. 
Let $H \subseteq \mathcal H$  be a subgraph of $\cH$ 
 and let $U_1$, $U_2$ be disjoint subsets of $V(H)$ with $|U_1|=|U_2|=m\ge \frac{n}{\ln^c n}$ such that $\deg_{U_i}(u)\ge \left(\frac{1}{2^{r-1}}+\gamma\right)p\binom{|U_i|}{r-1}$ for every $u\in U_j\neq U_i$ and $i=1,2$. Then there exists a 
  $(U_1,U_2)$-matching in $H$.
\end{lemma}

\begin{proof}
We choose with foresight $\eps\le 2^{r-3}\gamma$ such that $1+\eps<c$ holds.

For $T\subseteq U_1$ we define the neighbourhood $N(T)\subseteq U_2$ as follows (and similarly one defines $N(T)\subseteq U_1$ for $T\subseteq U_2$)
\[
N(T)=\{b\colon \exists\, a\in T,\, e\in H\text{ with }a, b\in e, |e\cap U_2|=r-1\}. 
\]

It will be sufficient to verify $|N(T)|\ge |T|$ for sets $T\subseteq U_i$ (Hall's condition) with $|T|\le \lceil m/2\rceil$ where $i\in[2]$. Assume w.l.o.g.\ that $T\subseteq U_1$ and further suppose towards a contradiction that $|N(T)|<|T|\le \lceil m/2\rceil$. Then, by the assumptions of the lemma, we have 
\[
e_H\left(T,\binom{U_2}{r-1}\right)
=\sum_{u\in T}\deg_{U_2}(u)\ge |T|\left(\frac{1}{2^{r-1}}+\gamma\right)p\binom{m}{r-1}.\]
 On the other hand, it follows from the $(\eps,p)$-pseudorandomness property~\ref{psrand:i} of $\cH$, that 
\[
e_H\left(T,\binom{U_2}{r-1}\right)=e_H\left(T,\binom{N(T)}{r-1}\right)\le (1+\eps)p|T|\binom{|N(T)|}{r-1}+|T|\ln^{1+\eps} n.
\] 
We estimate further 
$(1+\eps)p|T|\binom{|N(T)|}{r-1}< (1+\eps)p|T|\binom{\lceil m/2\rceil}{r-1}
\le \frac{1+2\eps}{2^{r-1}}p|T|\binom{m}{r-1}$. Comparing 
$\eps 2^{2-r}p|T|\binom{m}{r-1} +|T|\ln^{1+\eps} n$ with $\gamma p|T|\binom{m}{r-1}$, we obtain a contradiction in view of the choice of $\eps$ and the estimate $p\binom{m}{r-1}>\frac{\ln^{c r} n}{n^{r-1}} \frac{m^{r-1}}{r^{r-1}}\ge \frac{\ln^c n}{r^{r-1}}$.
\end{proof}

Analogously to Lemma~\ref{lem:matching}, one may prove almost verbatim the next lemma which asserts 
that between two disjoint subsets $U_1$, $U_2$ with $|U_2|/2\geq |U_1|=m=n/\polylog(n)$ and with all vertices in $U_1$ having high 'minimum degree' there must always be a $(U_1,U_2)$-matching $M$ consisting 
only of $(1,r-1)$-edges for $(U_1,U_2)$.  

\begin{lemma}\label{lem:matching2}
For every integer $r\geq 3$, every real $\gamma \in(0,1)$ and $c\geq 2$ there exists an $\eps>0$ such that  the following holds for any $(\eps,p)$-pseudorandom $r$-uniform hypergraph $\cH$ on $n$ vertices with  $n$ sufficiently large and $p \geq \frac{\ln^{c r} n}{n^{r-1}}$. 
Let $H \subseteq \mathcal H$  be a subgraph of $\cH$ 
 and let $U_1$, $U_2$ be disjoint subsets of $V(H)$ with $|U_2|/2\geq |U_1|=m\ge \frac{n}{\ln^c n}$ such that $\deg_H(u,U_2)\ge \left(\frac{1}{2^{r-1}}+\gamma\right)p\binom{|U_2|}{r-1}$ for every $u\in U_1$. 
Then there exists a 
  $(U_1,U_2)$-matching in $H$ consisting 
only of $(1,r-1)$-edges for $(U_1,U_2)$. \qed
\end{lemma}

%

The next technical definition is very handy to describe the basic structures we will be interested in.
\begin{definition}\label{def:two_match}
Given an $r$-uniform hypergraph $H$ and two disjoint sets $A$ and $B$ with $|B|\geq 2|A|$. A $2$-matching for $(A,B)$ is 
a collection of pairs of edges $(e_a,f_a)_{a\in A}$ so that 
\begin{enumerate}[(i)]
\item all these edges are distinct, 
\item $a\in e_a$, $a\in f_a$ for all $a\in A$,
\item there is an injection $\tau\colon \cup_{a\in A}\{e_a,f_a\}\to B$ with $\tau(g)\in g$, and
\item for every edge $g\in  \cup_{a\in A}\{e_a,f_a\}$: $|g\cap A|=1$ and $|g\cap B|=r-1$.
\end{enumerate}
\end{definition}

The next lemma allows us to find a $2$-matching.
\begin{lemma}\label{lem:doubling}
For every integer $r\geq 3$, every real $\gamma \in(0,2^{1-r})$ and $c>2$ there exists an $\eps>0$ such that  the following holds for any $(\eps,p)$-pseudorandom $r$-uniform hypergraph $\cH$ on $n$ vertices with $p \geq \frac{\ln^{c r} n}{n^{r-1}}$ and $n$ sufficiently large with $\Delta_2(\mathcal{H})\leq 2\ln n$. Let $H \subseteq \mathcal H$  be a subgraph of $\cH$  and let $A$ and $B$ be  disjoint subsets of $V(H)$ with 
$|A|=m$ and $|B|\geq 4m$, where $m\ge \frac{n}{\ln^c n}$, and such that 
$\deg_{H}(a,B)\ge \left(\frac{1}{2^{r-1}}+2.5\gamma\right)p\binom{|B|}{r-1}$ for every $a\in A$. Then there exists a $2$-matching for $(A,B)$.
\end{lemma}
\begin{proof}
We apply Lemma~\ref{lem:sampling} to $B$ and obtain an equipartition into $B_1\cup B_2$ such that $\deg_{H}(a,B_i)\ge \left(\frac{1}{2^{r-1}}+2\gamma\right)p\binom{|B_i|}{r-1}$ for every $a\in A$ and $i\in[2]$. An application of Lemma~\ref{lem:matching2} to $(A,B_1)$ and $(A,B_2)$ yields the desired $2$-matching.
\end{proof}

\section{Absorbers}\label{sec:absorbers}
A \emph{weak Berge path} (or simply \emph{weak path}) is an alternating sequence $(v_1, e_1, v_2, \ldots, v_k)$ of distinct vertices $v_1, \ldots, v_k$ and (not necessarily distinct) hyperedges $e_1, \ldots, e_{k-1}$ such that $v_i, v_{i+1} \in e_i$ for every $i\in [k-1]$. A weak path is called \emph{Berge path} if all its hyperedges are distinct.

For a weak  path $P=(v_1, e_1, \ldots, e_{k-1}, v_k)$ we denote by $E(P):= \{e_1, \ldots, e_{k-1}\}$ the set of hyperedges of $P$, by $\vertices (P):= \{v_1, \ldots, v_k\}$ the set of \emph{inner} vertices in the sequence of $P$, and by $V(P) := \bigcup_{i\in [k-1]} e_i$ the union of the hyperedges of $P$.
We say that $P$ \emph{connects} $v_1$ to $v_k$ and call $v_1$ and $v_k$ \emph{endpoints} of $P$. For a weak or Berge cycle $C=(v_1, e_1, v_2, \ldots, v_k, e_k)$ we define $\vertices (C):= \{v_1, \ldots, v_k\}$ and we refer to $(v_1,\ldots,v_k)$ as a \emph{sequence} of $C$.

The \emph{length} of a weak path $P$ is defined as $|\vertices(P)|-1$, and
the \emph{length} of a weak cycle $C$ is defined as $|\vertices(C)|$. In particular, if $P$ is a Berge path, then the length of $P$ is exactly the number of hyperedges of $P$. 
 Given two weak paths $P=(v_1, e_1, \ldots, e_{k-1}, v_k)$ and $Q=(v_k, e'_1, \ldots, e'_{k'-1}, v'_{k'})$ with $|\vertices(P) \cap \vertices(Q)|= 1$, we denote by $P\cdot Q$ the weak path 
 $(v_1, e_1, \ldots, e_{k-1}, v_k, e'_1, \ldots, e'_{k'-1}, v'_{k'})$.

We say that two Berge paths $P= (v_1, e_1, \ldots, e_{k-1}, v_k)$ and $P'= (v'_1, e'_1, \ldots, e'_{k'-1}, v'_{k'})$ are \emph{edge-disjoint} if $e_i \neq e_j$ for all $i\in[k-1]$ and $j\in [k'-1]$. \medskip

Next we introduce the notion of an absorber. 
\begin{definition}[Absorber for a vertex $u$]\label{def:absorber}
Given a (uniform) hypergraph $H$ and a vertex $u$. An absorber for $u$ is a subgraph $A$ of $H$ which consists of the following edges specified in the properties below:
\begin{enumerate}[(i)]
\item $A$ contains a Berge cycle $C$ with $u\in V^*(C)$ of length $2t+1$ for some $t\in\NN$ and with vertex sequence
$(u,v_1,\ldots,v_{t+1},\ldots, v_{2t})$;
\item there are $t-1$  Berge paths $P_1$, \ldots, $P_{t-1}$ so that each path $P_i$ has endpoints $v_{i+1}$ and $v_{2t+1-i}$ and the inner vertex sets are pairwise disjoint; 
\label{def:absorber:paths}
\item the edge-sets $E(C)$, $E(P_1)$,\ldots, $E(P_{t-1})$ are pairwise disjoint;
\item $E(A)=E(C)\cup\bigcup_{i\in [t-1]}E(P_i)$.
\end{enumerate}
We call the vertex $u$ a \emph{reservoir vertex} and the absorber $A$ a \emph{$u$-absorber}. The \emph{inner vertices} of $A$ are the vertices from $V^*(C)\cup\bigcup_i V^*(P_i)$. The vertices $v_1$ and $v_{t+1}$ are referred to as the \emph{main endpoints} of $A$.
\end{definition}

The following proposition about an absorber for some vertex $u$ explains its extreme usefullness in what comes and also the role of the main endpoints of an absorber -- the absorber $A_u$ contains two Berge paths with the same endpoints and the  inner vertices of the first path consist of all inner vertices of the absorber, while the inner vertices of the second contain all inner vertices but $u$.
\begin{proposition}\label{prop:absorber}
Let $A$ be a $u$-absorber in some hypergraph $H$ and let $C$ and $P_1$,\ldots, $P_{t-1}$ be the Berge cycle and Berge paths of the absorber $A$ respectively. Let  $(u,v_1,\ldots,v_{t+1},\ldots, v_{2t})$ be the sequence of $C$ according to property (i) in Definition~\ref{def:absorber}. Then $A$ contains the following two Berge paths:
\begin{enumerate}[(a)]
\item a Berge path $P_u$ from $v_1$ to $v_{t+1}$ with $V^*(P_u)=V^*(C)\cup\bigcup_{i\in [t-1]}V^*(P_{i})$, and \label{eq:absorber:path}
\item a Berge path $P$ from $v_1$ to $v_{t+1}$ with $V^*(P)=(V^*(C)\setminus\{u\})\cup\bigcup_{i\in [t-1]}V^*(P_{i})$.
\end{enumerate}
\end{proposition}
\begin{proof}
We will construct  (weak) paths $P_u$ and $P$ as described above. Since the Berge paths $P_i$ and the cycle $C$ in the absorber use different edges, it will imply that these weak paths are indeed Berge.

Let the structure of the cycle $C$ be as follows: $C=(u,e_1,v_1,e_2,\ldots,e_{t+1},v_{t+1},\ldots, v_{2t}, e_{2t+1})$. 
W.l.o.g.\  assume $t$ is even (the case $t$ odd is very similar). We construct $P_u$ as follows: 
\[
P_u=(v_1,e_1,u,e_{2t+1},v_{2t})\cdot P_1 \cdot (v_2,e_3,v_3) \cdot P_2\cdot \ldots\cdot P_{t-1} \cdot (v_t,e_{t+1},v_{t+1}).
\] 
Then the path $P$ is defined as follows:
\[
P=(v_1,e_2,v_2) \cdot P_1 \cdot (v_{2t},e_{2t-1},v_{2t-1}) \cdot P_2\cdot \ldots\cdot P_{t-1} \cdot (v_{t+2},e_{t+1},v_{t+1}).
\] 
It is most instructive to draw a picture: placing the inner vertices consecutively on a cycle and connecting appropriate vertices with the paths, one sees immediately that there is one way to traverse all vertices and a `complementary' way to traverse all vertices except for $u$.
\end{proof}

\section{Connection lemma}\label{sec:connection}
In this section we will concentrate on a connection lemma that will allow us to put Berge paths together into a longer Berge path. 

\subsection{An expansion lemma}

The following lemma allows us to prove an expansion property for a pseudorandom hypergraph in a resilience setting between any two `random', not too small vertex subsets.

\begin{lemma}\label{lem:simply_expand}
For every integer $r\geq 3$, every real $\gamma \in(0,2^{1-r})$ and $c>1$ there exists an $\eps>0$ such that  the following holds for any $(\eps,p)$-pseudorandom $r$-uniform hypergraph $\cH$ with $p \geq \frac{\ln^{c r} n}{n^{r-1}}$. Let $H \subseteq \mathcal{H}$  be a subgraph of $\cH$  and let $U_1$, $U_2$ be disjoint subsets of $V(H)$ with $|U_1|=|U_2|=m\ge \frac{n}{\ln^c n}$ such that $\deg_{H}(u,U_2)\ge \left(\frac{1}{2^{r-1}}+2\gamma\right)p\binom{m}{r-1}$ for every $u\in U_1$. Then for every subset $T_1\subseteq U_1$ of cardinality at least $\gamma m$ there exists a subset 
$T_2\subseteq U_2$ of cardinality at least $(1/2+\gamma) m$ such that for every $b\in T_2$ there exists a $(1,r-1)$-edge  $e$ for $(T_1,T_2)$ with 
$b\in e$. 
\end{lemma}
\begin{proof}
We choose $\eps>0$ such that $(1+\eps)(\tfrac{1}{2}+\gamma)^{r-1}<2^{1-r} +2\gamma$. 

Let an $(\eps,p)$-pseudorandom $r$-uniform hypergraph $\cH$ and a subhypergraph $H\subseteq \cH$ with sets $U_1$, $U_2$ as specified in the assumption of the lemma be given. 
Without loss of generality
let $T_1\subseteq U_1$ with $|T_1|=\gamma m$. We define $T_2:=\{b\colon \exists\, a\in T_1,\, e\in H\text{ with }a, b\in e, |e\cap U_2|=r-1\}$ and assume that $|T_2|<(1/2+\gamma)m$. First we arbitrarily extend $T_2$ to a subset $T_2'\subseteq U_2$ of size $(1/2+\gamma) m$.

Then compare the lower bound
$e_H(T_1,\binom{T_2}{r-1})
	=e_H(T_1,\binom{U_2}{r-1})
	\ge  |T_1|\left(\frac{1}{2^{r-1}}+2\gamma\right)p\binom{m}{r-1}$ 
with the upper bound on $e_H(T_1,\binom{T'_2}{r-1})$  which comes from condition~\ref{psrand:ii} of the definition of $(\eps,p)$-pseudorandomness (it is easily seen that $|T_2'|\ge \left(\frac{13 (r-1)!\ln n}{\eps^{3}p}\right)^{1/(r-1)}$ holds for $n$ large enough):
\begin{align*}
|T_1|\left(\frac{1}{2^{r-1}}+2\gamma\right)p\binom{m}{r-1}\le e_H \left(T_1,\binom{T_2}{r-1}\right)
&\le e_H \left(T_1,\binom{T'_2}{r-1}\right) \\
& \le  (1+\eps)p\gamma m\binom{(1/2+\gamma) m}{r-1}.
\end{align*}
Since $(1+\eps)p\gamma m\binom{(1/2+\gamma) m}{r-1}<(1+\eps)(1/2+\gamma)^{r-1}p\gamma m\binom{m}{r-1}$, we however obtain
a contradiction due to our choice of $\eps$. 
Thus, $T_2$ must contain more than $(1/2+\gamma)m$ vertices. 
\end{proof}

\subsection{A connection lemma}
Our goal is to prove a connection lemma, Lemma~\ref{lem:targeted}, which will allow us to connect a given collection of pairs of vertices by edge-disjoint Berge paths,  all of whose inner vertices but possibly endpoints are pairwise disjoint as well. The lemma is similar to the connection lemma from~\cite{ferber2014robust}. 

First we argue  that in a subgraph of a pseudorandom hypergraph there exists a Berge path between a sequence of disjoint, not so small sets, even after deleting a positive proportion of the vertices from each of the sets, if one has sufficiently high minimum vertex degree between every two `consecutive' sets.

\begin{lemma}\label{lem:robustly_connect}
For every integer $r\geq 3$, every real $\gamma \in(0,2^{1-r})$ and $c>1$ there exists an $\eps>0$ such that  the following holds for any $(\eps,p)$-pseudorandom $r$-uniform hypergraph $\cH$ on $n$ vertices with $p \geq \frac{\ln^{c r} n}{n^{r-1}}$ and $n$ sufficiently large. Let $H \subseteq \mathcal H$  be a subgraph of $\cH$  and let $U_1$, $U_2$, \ldots, $U_k$ (for some $k\ge \log_2 m$) be pairwise disjoint subsets of $V(H)$ each of cardinality $m\ge \frac{n}{\ln^c n}$ such that $\deg_{H}(u,U_{i+1})\ge \left(\frac{1}{2^{r-1}}+2\gamma\right)p\binom{m}{r-1}$ for every $u\in U_{i}$ and every $i\in[k-1]$. Then for every sequence $W_1,\ldots,W_k$ 
of subsets $W_i\subseteq U_i$, $i\in[k-1]$,  
such that $|W_1|=\gamma m$ and $|W_i|\ge (1-\gamma)m$ for $i\ge 2$, 
there exists a vertex $v_1\in W_1$ and a set $T_k\subseteq U_k$ of cardinality at least $(1/2+\gamma)m$ with the following property:
\begin{itemize}
\item for every $v_k\in T_k$ there is a Berge path $P=(v_1,e_1,\ldots,e_{k-1},v_k)$ with $v_i\in W_i$ for all $i\in [k-1]$ and such that  $|e_i\cap U_{i+1}|=r-1$ for all $i\in[k-1]$.
\end{itemize}

\end{lemma}
\begin{proof}
We choose $\eps>0$ so that Lemma~\ref{lem:simply_expand} is applicable on input $r$, $\gamma$ and $c$.

First we show the following statement for all $j=2,\ldots,k$:
\begin{itemize}
\item[($\star$)] there exists a subset $T_1\subseteq W_1$ of cardinality $\max\{1,\gamma m/2^{j-1}\}$ and 
a subset $T_j\subseteq U_j$ of cardinality at least $(1/2+\gamma) m$ such that for every $v_j\in T_j$ there is a Berge path $P=(v_1,e_1,\ldots,e_{j-1},v_j)$ with $v_1\in T_1$, $v_i\in W_i$ for all $i\in [j-1]$ and such that  $|e_i\cap U_{i+1}|=r-1$ for all $i\in[j-1]$.
\end{itemize}

The case $j=2$ follows from Lemma~\ref{lem:simply_expand} (all of whose assumptions are met): the Berge paths correspond to single edges with appropriate endpoints. 

Let now $j>2$ and proceed by induction. 
Then, we use the truth of statement ($\star$)
for $j-1$ and let $T_{1}$ and $T_{j-1}$ be the corresponding sets. We have $|T_{j-1}\cap W_{j-1}|\ge m/2$. 
We partition $T_1$ into two equal-sized sets $T_1'$ and $T_1''$ (if $|T_1|$=1 then we set $T_1':=T_1'':=T_1$) and, by the ($\star$)-property above, every vertex from $T_{j-1}\cap W_{j-1}$ is an endpoint of some Berge path  starting in one of the sets $T_1'$, $T_1''$. Therefore, we find a subset $T'_{j-1}\subseteq T_{j-1}\cap W_{j-1}$ with $|T'_{j-1}|\ge |T_{j-1}\cap W_{j-1}|/2\ge \gamma m$ and $T\in\{T_1',T_1''\}$ such that for every $v_{j-1}\in T'_{j-1}$ there is a Berge path $P=(v_1,e_1,\ldots,e_{j-2},v_{j-1})$ with  $v_i\in W_i$ for all $i\in [j-2]$ such that  $|e_i\cap U_{i+1}|=r-1$ for all $i\in[j-2]$ and, additionally, the vertex $v_1$ is  from $T$. Again, an application of Lemma~\ref{lem:simply_expand} allows us to extend these Berge paths such that the subset $T_j\subseteq U_j$ of all possible endpoints has cardinality at least $(1/2+\gamma)m$. We also have $|T_1'|=\max\{1,|T_1|/2\}$ and thus: $|T_1'|=\max\{1,\gamma m/2^{j-1}\}$. 

Since $\log_2(\gamma m)\le \log_2m-r+1\le k$ the statement of the lemma follows. 
\end{proof}

We can now apply the above lemma iteratively, obtaining $\gamma m$ edge-disjoint Berge paths. The following corollary summarizes it in a `symmetric' version.

\begin{corollary}\label{cor:set_of_paths}
For every integer $r\geq 3$, every real $\gamma \in(0,2^{1-r})$ and $c>1$ there exists an $\eps>0$ such that  the following holds for any $(\eps,p)$-pseudorandom $r$-uniform hypergraph $\cH$ on $n$ vertices with $p \geq \frac{\ln^{c r} n}{n^{r-1}}$ and $n$ sufficiently large. Let $H=(V,E) \subseteq \mathcal H$  be a subgraph of $\cH$. Let 
 $U_1$, $U_2$, \ldots, $U_{t}$, \ldots, $U_{t+t'}$ (for some $t,t'\ge \log_2 m$) be pairwise disjoint subsets of $V(H)$ each of cardinality $m\ge \frac{n}{\ln^c n}$. Further let $A$ and $B$ be two (not necessarily disjoint) sets of cardinality at least $3\gamma m$ each, disjoint from the other sets $U_i$. Assume moreover that the following conditions on the vertex degrees are satisfied
\begin{enumerate}[(i)]
\item $\deg_{H}(u,U_1)\ge \left(\frac{1}{2^{r-1}}+2\gamma\right)p\binom{m}{r-1}$ for every $u\in A$,
\item $\deg_{H}(u,U_{t+t'})\ge \left(\frac{1}{2^{r-1}}+2\gamma\right)p\binom{m}{r-1}$ for every $u\in B$,
\item $\deg_{H}(u,U_{i+1})\ge \left(\frac{1}{2^{r-1}}+2\gamma\right)p\binom{m}{r-1}$ for every $u\in U_{i}$ and every $i\in[t]$,
\item $\deg_{H}(u,U_{t+t'-i})\ge \left(\frac{1}{2^{r-1}}+2\gamma\right)p\binom{m}{r-1}$ for every $u\in U_{t+t'+1-i}$ and every $i\in[t']$,
\end{enumerate}
Then for any ordering of $A$ as $a_1$, $a_2$, \ldots and $B$ as $b_1$, $b_2$,\ldots, 
there exists a system of pairwise edge-disjoint Berge paths $P_1$, \ldots, $P_{\gamma m}$ such that 
\begin{enumerate}[(a)]
\item there exists a $(\gamma m)$-set $I=\{i_1,\ldots, i_{\gamma m}\}$ so that 
the endpoints of $P_j$ are $a_{i_j}$ and $b_{i_j}$ for every $j\in[\gamma m]$ and the edges lie completely within $\left(\cup_s U_s\right)\cup\{a_{i_j},b_{i_j}\}$,\label{item:Berge_P_i}
\item every vertex in $\left(\cup_s U_s\right)$ is an inner vertex of at most one
of the paths $P_j$,
\item every path $P_j$ has exactly one inner vertex from each of the sets $A$, $B$, $U_1$, \ldots, $U_{t+t'}$.\label{item:Berge_P_i_length}
\end{enumerate}
\end{corollary}
\begin{proof}
We choose $\eps>0$ so that Lemma~\ref{lem:robustly_connect} is applicable on input $r$, $\gamma$ and $c$. We will apply now Lemma~\ref{lem:robustly_connect} to find at least $\gamma m$ Berge paths $P_i$ as described in~\ref{item:Berge_P_i}--\ref{item:Berge_P_i_length}. 

Assume that we already found some paths 
$P_{1}$,\ldots, $P_{s}$ ($s<\gamma m$), and let $D_1$, \ldots, $D_{t+t'}$ be the sets of inner vertices of theses paths contained in the sets $U_1$, \ldots, $U_{t+t'}$. To construct a new Berge path, we apply Lemma~\ref{lem:robustly_connect} twice: to some $\gamma m$-subset $W_0$ of $A\setminus \{a_{i_j}\colon j\in[s]\}$ and 
the sets $W_i:=U_i\setminus D_i$ ($i\in [t]$) and to some $\gamma m$-subset $W_{t+t'+1}$ of $B\setminus \{b_{i_j}\colon j\in[s]\}$ and the sets $W'_i:=U_{t+t'-i+1}\setminus D_{t+t'-i+1}$ ($i\in [t'+1]$). This yields vertices $v_0\in W_0$ and $v_{t+t'+1}\in W_{t+t'+1}$ and sets $T_t\subseteq U_t$ and $T_{t}'\subseteq U_t$, each of size at least $(1/2+\gamma) m$ so that 
for every $v_t\in T_t$ there exists a Berge path that starts in $v_0$ and ends in $v_t$ and the inner vertices of which avoid already used vertices from the sets $D_i$, and such that for every $v_t\in T_t'$ there exists a Berge path with similar properties ending in $v_{t+t'+1}$. 
We may assume that $v_0=a_j$ and $v_{t+t'+1}=b_j$ for some $j$, not previously used, since we can apply Lemma~\ref{lem:simply_expand} 
to any $\gamma m$-subsets $W_0,W_{t+t'+1}$ of $A$ and $B$ 
and thus 
less than $\gamma m$ vertices of $A$ and of $B$ will fail to serve as a `starting vertex'.

Since $|T_t\cap T_t'|\geq 2\gamma m$
and as we used less than $\gamma m$ vertices from each of these sets for the inner vertices of the Berge paths $P_1$,\ldots, $P_s$, we finally obtain a Berge path
$P_j$ connecting $a_j$ and $b_j$ as required
for the properties (a)-(c).
Iterating  this procedure yields the desired system of $\gamma m$ Berge paths.
\end{proof}

Next we show how iterating the above corollary will allow us to connect $a_i$ with $b_i$ for every $i$.

\begin{lemma}[Connecting lemma]\label{lem:targeted}
For every integer $r\geq 3$, every real $\gamma \in(0,2^{1-r})$ and $c>7$ there exists an $\eps>0$ such that  the following holds for any $(\eps,p)$-pseudorandom $r$-uniform hypergraph $\cH$ on $n$ vertices with $p \geq \frac{\ln^{c r} n}{n^{r-1}}$, $\Delta_2({\mathcal H})\le 2\ln n$ and $n$ sufficiently large. Let $H=(V,E) \subseteq \mathcal H$ be a subgraph of $\cH$, let $A=\{a_1,\ldots,a_m\}$, $B=\{b_1,\ldots, b_m\}$ and $U$ be subsets of $V$ with 
\begin{enumerate}[(i)]
\item $(A\cup B)\cap U=\emptyset$,
\item $m\ge \frac{2n}{\ln^c n}$,
\item $|U|\ge \tfrac{3}{\gamma} m\log^2_2m$,
\item $\deg_{H}(u,U)\ge \left(\frac{1}{2^{r-1}}+3\gamma\right)p\binom{|U|}{r-1}$ for every $u\in A\cup B\cup U$,
\end{enumerate}
Then there exists a system of pairwise edge-disjoint Berge paths $P_1$, \ldots, $P_m$ such that 
\begin{enumerate}[(a)]
\item the endpoints of $P_i$ are $a_i$ and $b_i$ for every $i\in[m]$ and the edges lie completely within $U\cup\{a_i,b_i\}$ (if $a_i=b_i$ we abuse notation still calling $P_i$ a Berge path, although it is actually a Berge cycle), 
\label{prop1:connection}
\item every vertex in $U$ is an inner vertex of at most one
of the paths $P_i$,
\label{prop2:connection}
\item the length of each Berge path $P_i$ is between $2\log_2m+1$ and $4\log_2m+2$, and all lengths may be chosen to be even or odd at the same time.
\end{enumerate}
\end{lemma}
\begin{proof}
We choose $\eps>0$ so that Corollary~\ref{cor:set_of_paths} is applicable on input $r$, $\gamma$ and $c$. Notice that we may then apply Lemma~\ref{lem:sampling} on input $r$, $\gamma$, $c'=2.5$ and $c$.

 The proof strategy will proceed in rounds. In each round we will connect half of the yet not connected pairs $(a_i,b_i)$. It is clear, that after $\log_2 m$ rounds the process will terminate. Next we turn to the technical details. We describe the first two rounds and it will become clear how we proceed in the remaining rounds. We assume that all paths should have odd length (as the case of even length is treated similarly, by choosing an additional set $U'_{2\log_2 m+1}$ of size $m$ in every round).

 We come to the \emph{first round}. We use Lemma~\ref{lem:sampling} to  consecutively choose pairwise disjoint subsets $U'_1$, \ldots, $U'_{2\log_2m}$ 
 of size $m$ of the set $U$ such that the following holds (throughout the proof, the parameter $\ell=O(\log_2^2 m)$ is the number of subsets of size $m$ or $2m$ that we chose from $U$ so far):
\begin{itemize}
\item[(*)] $\deg_{H}(u,U')\ge \left(\frac{1}{2^{r-1}}+3\gamma-\tfrac{\ell}{\ln^{c'} n}\right)p\binom{|U'|}{r-1}$ for every $u\in\{a_1,\ldots,a_m\}\cup\{b_1,\ldots, b_m\}\cup U$ and every $U'\in\left\{U'_1,\ldots, U'_{2\log_2m}, U\setminus\left(\cup_{i=1}^{2\log_2 m}U'_i\right)\right\}$.
\end{itemize}
It is clear, that Corollary~\ref{cor:set_of_paths} is applicable. By performing this step $1/(2\gamma)$ times, i.e.\ by choosing each time new $2\log_2m$ pairwise disjoint sets with the property above, we can connect half of the pairs $(a_i,b_i)$ by desired Berge paths. Observe that in this first round we sampled (with Lemma~\ref{lem:sampling}) at most  $(2/\gamma)\log_2 m$ many pairwise disjoint $m$-subsets of $U$. We delete the vertices of the sampled sets from $U$, but use the same notation for simplicity. This finishes the first round.

 We move next to the \emph{second round}. Let $I$ be an $m/2$-set of indices of pairs $(a_i,b_i)$ which haven't been yet connected. We still denote the vertices $\{a_i\colon i\in I\}$ by $A$ and similarly for $B$. We apply Lemma~\ref{lem:sampling} to choose two disjoint subsets $A_1$ and $B_1$ of $U$, each of cardinality $2m$ such that 
\[
\deg_H(a,A_1)\ge (1-\tfrac{1}{\ln^{c'} n})\left(\frac{1}{2^{r-1}}+3\gamma-\tfrac{\ell}{\ln^{c'} n}\right)p\binom{|A_1|}{r-1} 
\text{  for all  }a\in A,
\]
and 
\[
\deg_H(b,B_1)\ge (1-\tfrac{1}{\ln^{c'} n})\left(\frac{1}{2^{r-1}}+3\gamma-\tfrac{\ell}{\ln^{c'} n}\right)p\binom{|B_1|}{r-1} 
\text{  for all  }b\in B.
\]
Then Lemma~\ref{lem:doubling} asserts the existence of $2$-matchings $(e_a,f_a)_{a\in A}$ and $(e_b,f_b)_{b\in B}$ for $(A,A_1)$, for $(B,B_1)$ respectively. Let $\tau_A$ and $\tau_B$ be the injections for these $2$-matchings, cf.\ Definition~\ref{def:two_match}. Next we order the vertices of $A$ according to the index set $I=\{i_1,\ldots,i_{m/2}\}$ as follows: 
$\tau_A(e_{a_{i_1}}), \tau_A(f_{a_{i_1}})$, $\tau_A(e_{a_{i_2}}), \tau_A(f_{a_{i_2}})$,\ldots, $\tau_A(e_{a_{i_{m/2}}}), \tau_A(f_{a_{i_{m/2}}})$. Exactly in the same way we order the vertices of $B$. We are now back in the original situation: 
\emph{exactly} as in the first round we sample 
at most $(2/\gamma)\log_2 m$ many pairwise disjoint $m$-subsets of $U$, and thus, we 
find a system of $m/2$ Berge paths $P^{(2)}_1$, \ldots, $P^{(2)}_{m/2}$ between the sets $A_1$ and $B_1$.  Out of $m/2$ Berge paths we find at least $m/4$ so that no two of them have endpoints of the form $\tau_A(e_{a_{i_j}})$, $\tau_A(f_{a_{i_j}})$ or $\tau_B(e_{b_{i_s}})$, $\tau_B(f_{b_{i_s}})$ for some $j$ or $s$ respectively. 

Recall, that we have $2$-matchings for $(A,A_1)$ and for $(B,B_1)$. Thus, using these we extend the $m/4$ Berge paths to $m/4$ Berge paths between $A$ and $B$ that satisfy properties~\ref{prop1:connection} and~\ref{prop2:connection} of the lemma. Observe that in this second round we sampled (with Lemma~\ref{lem:sampling}) disjoint subsets of $U$
the union of which contains at most $(\tfrac{2}{\gamma})m\log_2 m + 4m < (\tfrac{3}{\gamma})m\log_2 m$ vertices.

There remain $m/4$ vertices of $A$ and $B$ to be matched, and the corresponding $m/2$ vertices of $A_1$ and $B_1$, which are connected by $2$-matchings respectively. We proceed in a similar way in the next rounds by finding $2$-matchings between appropriate sets, finding $m/2$ Berge paths and then identifying $m/2^i$ additional Berge paths in the $i$th round. After $\log_2 m$ rounds we find the desired system of $m$ Berge paths by using at most $(3/\gamma)m\log_2^2 m$ vertices. This finishes the proof.
\end{proof}

\section{Proof of Theorem~\ref{berge:thm:main}}\label{sec:main_theorem}
\subsection{Proof outline}
We explain the idea of the proof for weak Hamilton cycles and not to worry about the Berge property.
 
 Given a spanning subhypergraph $H\subseteq \cH\sim \Hrnp$ as in the statement of the theorem, we partition the vertex set of $H$ into disjoint sets $Y$, $Z$ and  $W$, where $|Y|=o(n)$, $|Z|=n/\log^{\mathcal O(1)} n$ and  $W$ contains almost all vertices of $H$. The set $Z$ assumes the role of a reservoir. Choosing a partition with such sizes uniformly at random guarantees that with positive probability for every vertex $v$  the edges incident to $v$ are distributed as expected into the sets $Y$, $Z$ and $W$ (cf.\ Lemma~\ref{lem:sampling}). Also, since $H$ is a subgraph of the random hypergraph, we have good control on the edge distribution among various  subsets of vertices (cf.\ Lemmas~\ref{lem:upper_disc} and~\ref{lem:upper_disc_II}). 

Next we construct a weak path $Q$ with $\vertices(Q) \subseteq Y$ such that for every subset $M \subseteq Z$ there exists a weak path $Q_M$ that has the same endpoints as $Q$ and such that $\vertices(Q_M) = \vertices(Q) \dcup M$ (we will use   Lemma~\ref{lem:targeted} to construct absorbers as described by Proposition~\ref{prop:absorber} and to connect them into the  path $Q$). This property (\emph{absorbing property}) will be crucial at a later stage of the argument. 

Then we partition $W$ randomly into $\log^{\mathcal O(1)}n$ sets and distribute $Y\setminus \vertices(Q)$ among them such that all of these sets have the same size $o(n/\log^{\mathcal O(1)} n)$. Informally speaking, since $|Y|$ is significantly smaller than $|W|$ and every vertex from $Y$ is "well-connected" to $W$, such partition allows us to find weak paths $P_1, \ldots, P_m$ (using Lemma~\ref{lem:matching}), with $m=n/\log^{\mathcal O(1)} n$, so that $\vertices(P_1)$, \ldots, $\vertices(P_m)$ form a partition of $W\dcup Y\setminus\vertices(Q)$. 

As a last step, we use vertices from $Z$ to connect the paths $P_1$,\ldots, $P_m$ and $Q$ into a weak cycle $C$ (again this is possible since every vertex of $H$ is ``well-connected'' into $Z$). Since the unused vertices $M$ of $Z$ can be absorbed by the path $Q$ into a weak path $Q_M$ with $\vertices(Q_M)=\vertices(Q)\dcup M$, we have found a weak Hamilton cycle in $H$ in this way. To construct the path $Q$ and to connect the paths $P_1$,\ldots, $P_m$ and $Q$ into a cycle we will repeatedly use a lemma (connecting lemma, Lemma~\ref{lem:targeted}) that will allow us to connect various vertices by paths of length $\mathcal O(\log n)$. 

\subsection{Rigorous details: proof of Theorem~\ref{berge:thm:main}}
We show the following result about robustness of pseudorandom hypergraphs.
\begin{theorem}\label{berge:thm:main_psrandom}
For every integer $r\geq 3$ and every real $\gamma >0$ there exists an $\eps>0$ such that the following holds 
for any $(\eps,p)$-pseudorandom hypergraph $\cH$ on $n$ vertices with $\Delta_2(\cH)\le 2\ln n$, $p=\frac{\ln^{\kk} n}{n^{r-1}}$ and $n$ sufficiently large.  Let $H \subseteq \mathcal H$ be a spanning subgraph with $\delta_1(H) \ge \left(\frac{1}{2^{r-1}} + \gamma\right) p \binom{n}{r-1}$. Then $H$ contains a Hamilton Berge cycle. 
\end{theorem}
Observe first that Theorem~\ref{berge:thm:main_psrandom} implies immediately Theorem~\ref{berge:thm:main} for the probability $p=\frac{\ln^{\kk} n}{n^{r-1}}$, since a.a.s.\ the random hypergraph $\Hrnp$ is $(\eps,p)$-pseudorandom by Lemmas~\ref{lem:upper_disc} and~\ref{lem:upper_disc_II} and also satisfies $\Delta_2(\cH)\le 2\ln n$ a.a.s., by Proposition~\ref{prop:codegree}. The  following easy proposition briefly shows how the statement then extends to all $p\ge \frac{\ln^{\kk} n}{n^{r-1}}$ in a straightforward way.
\begin{proposition}\label{prop:sparsification}
If Theorem~\ref{berge:thm:main} holds for $p=\frac{\ln^{c} n}{n^{r-1}}$ then it is also true for $p\ge \frac{\ln^{c} n}{n^{r-1}}$.
\end{proposition}
\begin{proof}
Let $\cH\sim \Hrnp$ for some $p\ge \frac{\ln^{c} n}{n^{r-1}}$ and let $H\subseteq \cH$ be a subgraph of $\cH$ with minimum  vertex degree at least $\left(\frac{1}{2^{r-1}} + \gamma\right) p \binom{n}{r-1}$. We set $q:=\frac{\ln^{c} n}{p\cdot n^{r-1}}$ and denote by $G_q$ the random subgraph of $G$ where each edge is kept with probability $q$ independently of the other edges. Clearly, $\cH_q\sim H^{({r})}\left(n,qp\right)$ and $H_q\subseteq \cH_q$ and, by Chernoff's inequality (Theorem~\ref{chernoff1}) a.a.s.\ $\delta_1(H_q)\ge \left(\frac{1}{2^{r-1}} + \gamma/2\right) pq \binom{n}{r-1}$, and $\Delta_2(\cH_q)\le 2\ln n$ a.a.s.\ as well (Proposition~\ref{prop:codegree}). Hence, we may apply Theorem~\ref{berge:thm:main} in the special case when $pq=\frac{\ln^{c} n}{n^{r-1}}$ and the general claim follows.
\end{proof}
\begin{proof}[Proof of Theorem~\ref{berge:thm:main_psrandom}]\mbox{}\\
\noindent{}\textbf{Setup.} 
 W.l.o.g.\ we assume that $\gamma<2^{3-r}$ and we set $\gamma'=\gamma/4$. The following auxiliary parameters 
 $c_1:=4$, $c_2:=c_1+3$, $c':=7.4$ and $c=17$ are given here for future reference only and the aim of their use is not to obscure the technical details. We thus have $p=\frac{\ln^{cr} n}{n^{r-1}}$. Let $\eps_1$ be as asserted by Lemma~\ref{lem:matching} on input $r$, $\gamma'$ and $c$, let $\eps_2$  be as asserted by Lemma~\ref{lem:targeted} on input $r$, $\gamma'$ and $c$. We set $\eps:=\min\{\eps_1,\eps_2\}$.

  We apply Lemma~\ref{lem:sampling}  to $V(H)$ twice to obtain three pairwise-disjoint sets $W$, $Y$ and $Z$ with the following properties:
\begin{enumerate}[(i)]
\item $|Z|=\frac{n}{\log_2^{c_1} n}$ and $\deg(v,Z)\ge (2^{1-r}+4\gamma'-\tfrac{2}{\ln^{c'} n})p\binom{|Z|}{r-1}$ for all $v\in V$,
\item $|Y|=\tfrac{9}{\gamma'}|Z|\log_2^3n$ and $\deg(v,Y)\ge (2^{1-r}+4\gamma'-\tfrac{2}{\ln^{c'} n})p\binom{|Y|}{r-1}$ for all $v\in V$,
\item $|W|=V\setminus(Y\cup Z)$ and $\deg(v,W)\ge (2^{1-r}+4\gamma'-\tfrac{2}{\ln^{c'} n})p\binom{|W|}{r-1}$ for all $v\in V$.
\end{enumerate}

\noindent{}\textbf{Constructing an absorbing path $P_A$.} 
Our aim here is to  construct  absorbers for every $u\in Z$ and to put them into a single path. For this we will use vertices from $Y$, by applying Lemma~\ref{lem:sampling} followed by Lemma~\ref{lem:targeted} several times. More precisely, we construct for every $u\in Z$ a $u$-absorber $A_u$ such that the inner vertices of all $A_u$'s are pairwise disjoint and the edges are pairwise disjoint as well. We will do it in three stages. 

In the \emph{first stage} we apply Lemma~\ref{lem:sampling} to the set $Y$ and obtain a set $U_1$ of cardinality 
$\tfrac{3}{\gamma'}|Z|\log^2_2n$ such that 
\begin{enumerate}[(i)]
\item $\deg_H(v,U_1)\ge (2^{1-r}+3\gamma') p\binom{|U_1|}{r-1}$ for all $v\in V$, and
\item $\deg_H(v,Y\setminus U_1)\ge  (2^{1-r}+4\gamma'-\tfrac{3}{\ln^{c'}n}) p\binom{|Y|-|U_1|}{r-1}$ for all $v\in V$.
\end{enumerate}
We then apply Lemma~\ref{lem:targeted} (with $A,B=Z$ and $U=U_1$) to put each vertex $u\in Z$ on its own Berge cycle $C_u$ of \emph{odd} length between $2\log_2 |Z|+1$ and $4\log_2 |Z|+1$, in order to obtain the cycles
needed by Definition~\ref{def:absorber}.

 In the \emph{second stage} we need to connect the corresponding pairs of vertices on each of the Berge cycles $C_u$, where $u\in Z$, as specified in the property~\ref{def:absorber:paths} from the definition of the asorbers, Definition~\ref{def:absorber}. Observe that there are at most $2\log_2 |Z|-1$ such pairs for each absorber, which requires connecting in total at most $2|Z|\log_2|Z|$ pairs. We apply Lemma~\ref{lem:sampling} to the set $Y\setminus U_1$ and obtain a set $U_2$ of cardinality 
$\tfrac{6}{\gamma'}|Z|\log^3_2n$ such that 
\begin{enumerate}[(i)]
\item $\deg_H(v,U_2)\ge (2^{1-r}+3\gamma') p\binom{|U_2|}{r-1}$ for all $v\in V$, and
\item $\deg_H(v,Y\setminus (U_1\cup U_2))\ge  (2^{1-r}+4\gamma'-\tfrac{4}{\ln^{c'}n})  p\binom{|Y|-|U_1|-|U_2|}{r-1}$ for all $v\in V$.
\end{enumerate}
Again, an application of the connecting lemma, Lemma~\ref{lem:targeted}, yields the system of Berge paths, that completes for each $u\in Z$ an absorber $A_u$ with the required properties.

 Finally, in the \emph{third stage}, we put all our absorbers onto a Berge path. More precisely, we aim to connect the paths of all $A_u$s as specified in Proposition~\ref{prop:absorber}~\ref{eq:absorber:path} into a Berge path. 
   We consider  the main endpoints $u_i$ and $u_i'$ of every $u$-absorber and we wish to connect $u'_i$ with $u_{i+1}$ for every $i\in[|Z|-1]$ by edge disjoint Berge paths (whose inner vertex sets are pairwise disjoint as well), using again the connection lemma, Lemma~\ref{lem:targeted}. Recall, that by our choice of $Y$, we have $|Y\setminus (U_1\cup U_2)|\ge \tfrac{4}{\gamma'}|Z| \log^2_2n$ and $\deg_H(v,Y\setminus (U_1\cup U_2))\ge  (2^{1-r}+3\gamma') p\binom{|Y|-|U_1|-|U_2|}{r-1}$ for all $v\in V$. Therefore, all assumptions of Lemma~\ref{lem:targeted} are met, and we obtain thus a system of Berge paths that connects the required pairs of endpoints of the absorbers $A_u$ ($u\in Z$) into the \emph{absorbing Berge path} $P_A$. Observe that $V^*(P_A)\subseteq Z\cup Y$, the endpoints of $P_A$ are  $u_1$ and $u'_{|Z|}$ and the path $P_A$ has the following \emph{absorbing property} (by Proposition~\ref{prop:absorber}):
\begin{itemize}
\item for every subset $Z'\subseteq Z$, there exists a Berge path $P_{Z'}$ with the same endpoints as $P_A$ and $V^*(P_{Z'})=V^*(P_A)\setminus Z'$.
\end{itemize}
We will later use some vertices from $Z$ for further connections that come and we denote by $Y'$ the vertices of $Y$ which are neither inner vertices of $P_A$.
 
\noindent{}\textbf{Partitioning $W$.} 
As a next step we partition $W$ into $\log_2^{c_2}n$ sets $W_1$, \ldots, $W_i$ of the same cardinality plus a remainig set $M$ of fewer than $n/\log_2^{c_2}n$ vertices. 
 We do so by consecutively applying Lemma~\ref{lem:sampling}, so that  the sets $W_i$ will satisfy the following properties:
\begin{itemize}
\item $\deg(v,W_i)\ge (2^{1-r}+3\gamma')p\binom{|W_i|}{r-1}$ for all $v\in V$.
\end{itemize}
Next we distribute the vertices from $Y'$ equally among the sets $W_i$'s (but we still use the same notation for these new sets) and put the at most $\log_2^{c_2}n$ vertices to $M$. Observe that then still $\deg(v,W_i)\ge (2^{1-r}+2\gamma')p\binom{|W_i|}{r-1}$ for all $v\in V$ holds, by the choice of $c_1$, $c_2$ and the sizes of $Y'$ and $W$.

\noindent{}\textbf{Covering $W$ with Berge paths.} 
Now we apply Lemma~\ref{lem:matching} to obtain $(W_i,W_{i+1})$-matchings in $H$ for $i=1,\ldots,\log_2^{c_2}n$. Notice 
that this gives rise to a system of $(1-o(1))n/\log_2^{c_2}n$ edge-disjoint Berge paths, each of length $\log_2^{c_2}n$, so that the inner vertices of these paths are pairwise disjoint and form a partition of $W\setminus M$. To obtain a particular path, one starts with 
some vertex $w_1\in W_1$, then follows the matching edge $e_1\ni w_1$, then considers the second endpoint $w_2 \in e_1$, then follows the matching edge $e_2\ni w_2$ from the $(W_2,W_{3})$-matching in $H$ and so on. 

Additionally, we also view every single vertex $w\in M$ as a Berge path. Thus, in total we cover $W\cup Y'$ by $t$ Berge paths
 $P_1$, \ldots, $P_t$,  
where $t$ is less than  $2(n/\log_2^{c_2}n)+\log_2^{c_2}n$.
\medskip

\noindent{}\textbf{Obtaining a weak Hamilton cycle in $H$.} 
Recall, that $|Z|=\frac{n}{\log_2^{c_1} n}$ and $\deg(v,Z)\ge (2^{1-r}+3\gamma')p\binom{|Z|}{r-1}$ for all $v\in V$ holds. Thus, the assumptions of the connecting lemma, Lemma~\ref{lem:targeted}, are met. Therefore we are able  
 to connect the Berge paths  $P_1$, \ldots, $P_t$ and the absorbing Berge path $P_A$ into some `almost' cycle $C'$ by using the vertices of some subset $Z'$ from $Z$.  By the absorbing property, we may delete the vertices from $Z'$ and restructure the path $P_A$ (which is a subgraph of $C'$) into $P'$ so that the so obtained cycle $C$ is indeed a weak Hamilton cycle.

\noindent{}\textbf{Why is $C$ already a  Hamilton Berge cycle in $H$.} 
The cycle which we constructed in the way above is indeed Berge. The reason is that we build our cycle by constructing Berge paths between partite sets (cf.\ Lemmas~\ref{lem:matching},~\ref{lem:matching2} and~\ref{lem:targeted}) and in doing so we use only edges $e$ between some two sets $U_i$ and $U_{i+1}$ which lie within $U_i\cup U_{i+1}$ such that $|e\cap U_i|\in\{1,r-1\}$. In this way we guarantee that at any stage of our construction we are using genuinely new edges. Thus, the constructed cycle is indeed Berge. 
\end{proof}

\section{Berge Hamiltonicity in dense hypergraphs}\label{sec:Dirac}
In this section we prove Proposition~\ref{thm:DiracBerge}. 
 The asymptotic tightness of the bound on the minimum vertex degree was considered already in the introduction.  We remark that optimal bound was proven in a long paper in~\cite{CP18}, whereas the proof below is elementary and very short.

\begin{proof}[Proof of Proposition~\ref{thm:DiracBerge}]
Let $r\geq 3$ and let $H=(V,E)$ be an $r$-uniform hypergraph on $n > 2r-2$ vertices with $\delta_1(H) \ge  \binom{\lceil n/2 \rceil -1}{r-1}+ n-1$. We observe first that due to the large minimum vertex degree,  $H$ is connected. Let $P = (v_1, e_1, v_2, \ldots, e_{k-1}, v_k)$ be a longest Berge path in $H$. 

For every $v\in V$ we define $E'(v) = \big\{e\in E\setminus \{e_1, \ldots, e_{k-1}\}:\, v\in e\big\}$. The condition on the minimum vertex degree implies that we have $|E'(v_1)|, \, |E'(v_k)| \geq \binom{\lceil n/2\rceil -1}{r-1}$. 
Since $P$ is a longest Berge path, it holds for every $e\in E'(v_1)$ that $e \subseteq \vertices(P)$. The same is true for $v_k$. 

We claim that there exist distinct hyperedges $e \in E'(v_1)$ and $e' \in E'(v_k)$ as well as an index $i\in [k-1]$ such that $v_{i+1}\in e \cap \vertices(P)$ and $v_i \in e' \cap \vertices(P)$. Assume for a contradiction that this is not true. Then, by the pigeonhole principle, there exists a subset $S\subseteq [k-1]$ such that $|S| \leq \lfloor (k-1)/2 \rfloor$ with $f\subseteq \{v_{i+1}\colon i\in S\}\cup\{v_1\}$ for every $f\in E'(v_1)$ or with $f'\subseteq \{v_{i}\colon i\in S\}\cup\{v_k\}$ for every $f' \in E'(v_k)$. Suppose that $f\subseteq \{v_{i+1}\colon i\in S\}\cup\{v_1\}$ for every $f\in E'(v_1)$ holds. Then $\delta_1(v_1) \leq \binom{|S|-1}{r-1} +k -1 < \binom{\lceil n/2\rceil -1}{r-1} + n-1$, which is a contradiction.

Hence, there exist $e \in E'(v_1)$ and $e' \in E'(v_k)$ with the claimed property. Let 
\[
C= (v_1, e, v_{i+1}, e_{i+1}, v_{i+2}, \ldots, v_k, e', v_i, e_{i-1}, \ldots, e_1)
\]
 be the Berge cycle that can be constructed from $P$ using $e$ and $e'$. If $k= n$, then $C$ is a Hamilton Berge cycle and we are done. Otherwise we get a contradiction, similar as in Dirac's original proof, by breaking up the cycle $C$ and extending it by a new edge (since $H$ is connected), thus obtaining a longer Berge path.
\end{proof}

\bibliographystyle{abbrv} 	 
\bibliography{BergeCycles}	 

\begin{thebibliography}{10}

\bibitem{allenBT}
P.~Allen, J.~B{\"o}ttcher, J.~Ehrenm\"uller, and A.~Taraz.
\newblock The bandwidth theorem in sparse graphs.
\newblock preprint arXiv:1612.00661.

\bibitem{BD18}
D.~Bal and P.~Devlin.
\newblock Hamiltonian {B}erge cycles in random hypergraphs.
\newblock preprint arXiv:1809.03596, 2018.

\bibitem{BalCsaSam11}
J.~Balogh, B.~Csaba, and W.~Samotij.
\newblock Local resilience of almost spanning trees in random graphs.
\newblock {\em {Random Struct. Algorithms}}, 38(1-2):121--139, 2011.

\bibitem{BLS12}
J.~Balogh, C.~Lee, and W.~Samotij.
\newblock Corr{\'a}di and {H}ajnal's theorem for sparse random graphs.
\newblock {\em Combin. Probab. Comput.}, 21(1-2):23--55, 2012.

\bibitem{bermond1976}
J.~C. {Bermond}, A.~{Germa}, M.~C. {Heydemann}, and D.~{Sotteau}.
\newblock {Hypergraphes hamiltoniens.}
\newblock {Probl\`emes combinatoires et th\'eorie des graphes, Orsay 1976,
  Colloq. int. CNRS No. 260, 39-43 (1978).}, 1978.

\bibitem{BKT13}
J.~B{\"o}ttcher, Y.~Kohayakawa, and A.~Taraz.
\newblock Almost spanning subgraphs of random graphs after adversarial edge
  removal.
\newblock {\em Combin. Probab. Comput.}, 22(05):639--683, 2013.

\bibitem{CEP16}
D.~{Clemens}, J.~{Ehrenm\"uller}, and Y.~{Person}.
\newblock {A Dirac-type theorem for Hamilton Berge cycles in random
  hypergraphs.}
\newblock In {\em {Discrete mathematics days. Extended abstracts of the 10th
  ``Jornadas de matem\'atica discreta y algor\'{\i}tmica'' (JMDA), Barcelona,
  Spain, July 6--8, 2016}}, pages 181--186. Amsterdam: Elsevier, 2016.

\bibitem{CP18}
M.~Coulson and G.~Perarnau.
\newblock A rainbow {D}irac's theorem.
\newblock preprint arXiv:1809.06392, 2018.

\bibitem{dirac1952}
G.~A. Dirac.
\newblock Some theorems on abstract graphs.
\newblock {\em Proceedings of the London Mathematical Society}, 3(1):69--81,
  1952.

\bibitem{FH18}
A.~Ferber and L.~Hirschfeld.
\newblock Co-degrees resilience for perfect matchings in random hypergraphs.
\newblock http://math.mit.edu/~ferbera/Resilience.pdf, 2018.

\bibitem{ferber2014robust}
A.~{Ferber}, R.~{Nenadov}, A.~{Noever}, U.~{Peter}, and N.~{\v{S}kori\'c}.
\newblock {Robust Hamiltonicity of random directed graphs.}
\newblock {\em {J. Comb. Theory, Ser. B}}, 126:1--23, 2017.

\bibitem{HStSu16}
D.~{Hefetz}, A.~{Steger}, and B.~{Sudakov}.
\newblock {Random directed graphs are robustly Hamiltonian.}
\newblock {\em {Random Struct. Algorithms}}, 49(2):345--362, 2016.

\bibitem{HLS12}
H.~Huang, C.~Lee, and B.~Sudakov.
\newblock Bandwidth theorem for random graphs.
\newblock {\em J. Comb.Theory Ser. B}, 102(1):14--37, 2012.

\bibitem{janson2011random}
S.~Janson, T.~{\L}uczak, and A.~Ruci{\'n}ski.
\newblock {\em Random graphs}, volume~45.
\newblock John Wiley \& Sons, 2011.

\bibitem{KLS10}
M.~Krivelevich, C.~Lee, and B.~Sudakov.
\newblock Resilient pancyclicity of random and pseudorandom graphs.
\newblock {\em SIAM J. Discrete Math.}, 24(1):1--16, 2010.

\bibitem{KO14}
D.~K{\"u}hn and D.~Osthus.
\newblock Hamilton cycles in graphs and hypergraphs: an extremal perspective.
\newblock In {\em Proceedings of the International Congress of Mathematicians,
  Seoul, Korea}, volume~4, pages 381--406, 2014.

\bibitem{lee2012dirac}
C.~Lee and B.~Sudakov.
\newblock Dirac's theorem for random graphs.
\newblock {\em Random Struct. Algorithms}, 41(3):293--305, 2012.

\bibitem{NST18}
R.~Nenadov, A.~Steger, and M.~Truji{\'c}.
\newblock Resilience of perfect matchings and {H}amiltonicity in random graph
  processes.
\newblock Random Structures \& Algorithms, online, 2018.

\bibitem{poole2014weak}
D.~Poole.
\newblock On weak {H}amiltonicity of a random hypergraph.
\newblock {\em arXiv:1410.7446}, 2014.

\bibitem{rodl2010dirac}
V.~R{\"o}dl and A.~Ruci{\'n}ski.
\newblock Dirac-type questions for hypergraphs --- a survey (or more problems
  for {E}ndre to solve).
\newblock In {\em An Irregular Mind}, pages 561--590. Springer, 2010.

\bibitem{RRSz06}
V.~R{\"o}dl, A.~Ruci{\'n}ski, and E.~Szemer{\'e}di.
\newblock A {D}irac-type theorem for 3-uniform hypergraphs.
\newblock {\em Combin. Probab. Comput.}, 15(1-2):229--251, 2006.

\bibitem{SST18}
N.~{\v{S}}kori{\'c}, A.~Steger, and M.~Truji{\'c}.
\newblock Local resilience of an almost spanning $k$-cycle in random graphs.
\newblock {\em Random Structures \& Algorithms}, 53(4):728--751, 2018.

\bibitem{SudVu}
B.~Sudakov and V.~H. Vu.
\newblock Local resilience of graphs.
\newblock {\em Random Struct. Algorithms}, 33(4):409--433, 2008.

\bibitem{zhao2015recent}
Y.~{Zhao}.
\newblock {Recent advances on Dirac-type problems for hypergraphs.}
\newblock In {\em {Recent trends in combinatorics}}, pages 145--165. Cham:
  Springer, 2016.

\end{thebibliography}
\end{document}